\documentclass[reqno,11pt]{amsart}

\usepackage{amsmath,amssymb,amsthm, epsfig}
\usepackage[usenames,dvipsnames]{pstricks}
\usepackage{epsfig}
\usepackage{pst-grad} 
\usepackage{pst-plot}

 \usepackage[usenames,dvipsnames]{pstricks}
 \usepackage{epsfig}
 \usepackage{pst-grad} 
 \usepackage{pst-plot} 
 \usepackage[space]{grffile} 
 \usepackage{etoolbox} 
 \makeatletter 
 \patchcmd\Gread@eps{\@inputcheck#1 }{\@inputcheck"#1"\relax}{}{}
 \makeatother

\def \R {\mathbb{R}}
\def \dist {\mathrm{dist}} 
 
\newtheorem{theorem}{Theorem} 
\newtheorem{lemma}{Lemma}
\newtheorem{proposition}{Proposition}  

\newtheorem{definition}{Definition}
\newtheorem{remark}{Remark}

\numberwithin{equation}{section}
\title[Infinity two-phase free boundary problem]{On a  two-phase free boundary problem ruled by the infinity Laplacian}

\author[D.J. Ara\'ujo]{Dami\~ao J. Ara\'ujo}
\address{Department of Mathematics, Universidade Federal da Para\'iba, 58059-900, Jo\~ao Pessoa-PB, Brazil}{}
\email{araujo@mat.ufpb.br}

\author[E. V. Teixeira]{Eduardo V. Teixeira}
\address{Department of Mathematics, University of Central Florida, 32816, Orlan\-do-FL, USA}{}
\email{eduardo.teixeira@ucf.edu}

\author[J.M. Urbano]{Jos\'e Miguel Urbano}
\address{University of Coimbra, CMUC, Department of Mathematics, 3001-501 Coimbra, Portugal \& Department of Mathematics, Universidade Federal da Pa\-ra\'\i \-ba, 58059-900 Jo\~ao Pessoa, PB-Brazil}
\email{jmurb@mat.uc.pt}

\begin{document}



\begin{abstract} 
In this paper we consider a two-phase free boundary problem ruled by the infinity Laplacian. Our main result states that bounded viscosity solutions in $B_1$ are universally Lipschitz continuous in $B_{1/2}$, which is the optimal regularity for the problem. We make a new use of the Ishii-Lions' method, which works as a surrogate for the lack of a monotonicity formula and is bound to be applicable in related problems.
\end{abstract}  

\date{\today}

\maketitle

\section{Introduction} \label{sct intro}

Let $n \ge 2$ and $\Omega \subset \mathbb{R}^n$ be a bounded domain. In this paper we are interested in the singular free boundary problem, ruled by the infinity Laplacian,

\begin{equation}\label{eq}
\left\{ 
\begin{array}{rcll} 
-\Delta_{\scriptscriptstyle \infty} u & = & {f_+} (x)& \ {\rm in} \  \{u>0\} { \, \cap \, \Omega} \\[0.2cm]
-\Delta_{\scriptscriptstyle \infty} u & = & {f_-} (x)& \ {\rm in} \  \{u<0\}{  \, \cap \, \Omega} \\[0.2cm]
\max\{u_\nu^+,u_\nu^-\} & \le & \Lambda & \ \mbox{on} \  \mathcal{R}(u),\\  
\end{array}
\right.  
\end{equation}
\medskip
where $\Lambda >0$, {  $f_+,f_- \in L^\infty (\Omega)$} are given, 
$$
\mathcal{R}(u):= \partial(\{u>0\} \cup \{u<0\}) {  \, \cap \, \Omega},
$$
and $u_\nu^{+}$ and $u_\nu^{-}$ represent the corresponding normal derivatives in a very weak sense to be described later. 
The governing operator is the so called infinity Laplacian,
$$
	 \Delta_\infty v :=  \left\langle D^2v D v, Dv \right\rangle.
$$

Of particular interest is the case when $- \Delta_\infty u$ takes two different constant (nonzero) values in each phase, say $f_{+} = - \chi_{\{ u > 0 \}}$ and $f_{-} = + \chi_{\{ u < 0 \}}$; compare it with the classical Prandtl-Batchelor theorem in fluid dynamics (cf. \cite{A,B}).

The main result we obtain is that {\it any} bounded viscosity solution of \eqref{eq}, in a sense to be detailed, is locally Lipschitz continuous.  We stress that Lipschitz estimates are sharp for such a free boundary problem, as simple examples show. We also note that any $C^2$ function with smooth zero-level set satisfies the  free boundary problem \eqref{eq}, for some $f_{+}$ and $f_{-}$. On the other hand, since $\mathcal{R}(u)$ is unknown, the gradient control to be proven in this article is far from being obvious or easy to obtain. For related issues, where specific bounds are prescribed on unknown sets and a PDE is given in the complementary regions, we refer, for example, to \cite{FS}.  

Indeed, while it is clear that a function satisfying $-\Delta_\infty u =f_{+} \chi_{\{u> 0 \}} + f_{-}\chi_{\{u< 0 \}} $ in $  \{u>0 \} \cup   \{u<0 \} $ is locally Lipschitz continuous in its phases, the corresponding estimates degenerate as one approaches their (unknown) boundaries. Thus, the main difficulty when proving the optimal regularity for our problem is the Lipschitz regularity across the free boundary.  We also comment that, in the case where $f_\pm = 0$, Lipschitz regularity immediately follows as a consequence of the findings in \cite{CEG}, as both $u^{+}$ and $u^{-}$ are then viscosity subsolutions of $-\Delta_\infty v = 0$.  

Our strategy for proving universal Lipschitz estimates for the two-phase problem \eqref{eq} relies on doubling variables, in the spirit of the Ishii-Lions' method \cite{IL}, in a fashion carefully designed to match the structure of the infinity Laplacian. See \cite{EY, JLM, LW} for more on this highly degenerate operator and also \cite{RTU}, for another free boundary problem involving it.

The paper is organized as follows. In the next section, we define precisely what we mean by a solution of \eqref{eq} and state our main result. The rest of the paper is devoted to its proof; in section 3 we derive pointwise estimates for interior maxima of a certain function, which will be instrumental in the sequel; section 4 brings the definition of an appropriate barrier; the proof is carried out in section 5 and ultimately amounts to the analysis of an alternative. 

\section{Definition of solution and main result}\label{regularity}

We will consider very weak solutions of problem \eqref{eq} for which we nevertheless obtain optimal regularity results. The appropriate notion is that of  viscosity solution and we need to first recall the definition of jet, given, e.g., in \cite{CIL}.

{Let $\Omega \subset \mathbb{R}^n$ be a bounded domain}, $u \colon \Omega \subset \R^n \rightarrow \R$ and $\hat{x} \in \Omega$. Denoting by $\mathcal{S} (n)$ the set of all $n\times n$ symmetric matrices, the second-order superjet of $u$ at $\hat{x}$, ${J}_{\Omega}^{2,+}u(\hat{x})$, is the set of all ordered pairs $(p,X) \in {\R}^n \times \mathcal{S} (n)$ such that
$$
u(x) \leq u(\hat{x}) + \langle p , x-\hat{x} \rangle + \frac12 \langle X (x-\hat{x}), x-\hat{x} \rangle + o\left( |x-\hat{x}|^2\right)
$$
as $\Omega \ni x \to \hat{x}$. The second-order subjet of $u$ at $\hat{x}$ is defined by ${J}_{\Omega}^{2,-}u(\hat{x}) := -{J}_{\Omega}^{2,+}(-u)(\hat{x})$. For $\hat{x} \in \overline{\Omega}$, we also denote by
$\overline{J}_{\Omega}^{2,\pm}u(\hat{x})$ the set of all pairs $(p,X) \in {\R}^n \times \mathcal{S} (n)$ for which there exist sequences $x_j \in \Omega$ and $(p_j,X_j) \in {J}_{\Omega}^{2,\pm}u(x_j)$, such that { $(x_j,p_j,X_j)\to (\hat{x},p,X)$, as $j \to \infty$}. 

We are now ready to disclose in what sense the equation and the free boundary condition in \eqref{eq} are to be interpreted. For the sake of simplicity, we shortly denote  $\{u>0\}:=\{x \in \Omega \, | \, u(x)>0\}$ (accordingly for $\{u<0\}$).  

\begin{definition}\label{viscositysol}
An  {upper} semi-continuous function $u$ is a \textit{viscosity subsolution of} \eqref{eq} in $\Omega$ if the following two conditions hold:

\begin{itemize}

\item[(i)] for each $x \in \{u>0\}$ (resp. $x \in \{u<0\}$) and $\left( \xi,M \right) \in { {J}_{\Omega}^{2,+}u(x)}$, we have
$$
-\langle M\xi,\xi \rangle \leq   {f_+} \quad (\mbox{resp.} -\langle M\xi,\xi \rangle \leq  {f_-});
$$
 
\item[(ii)] for each $x \in \mathcal{R}(u)$ and $(\xi,M) \in  {J}_{\Omega}^{2,+}u(x)$, {with $\xi \neq 0$}, we have
$$ 
u \left( x-t \frac{\xi}{|\xi|} \right) \geq -\Lambda t+o(t), \quad {\mbox{as } t \to 0^+}.
$$
\end{itemize}

A {lower} semi-continuous function $u$ is a \textit{viscosity supersolution of} \eqref{eq} in $\Omega$ if the following two conditions hold:

\begin{itemize} 

\item[(i)] for each $x \in \{u>0\}$ (resp. $x \in \{u<0\}$) and $\left( \xi,M \right) \in { {J}_{\Omega}^{2,-}u(x)}$, we have
$$
-\langle M\xi,\xi \rangle \geq  {f_+} \quad (\mbox{resp.} -\langle M\xi,\xi \rangle \geq  {f_-});
$$

\item[(ii)] for each $x \in \mathcal{R}(u)$ and $\left( \xi,M \right) \in {J}_{\Omega}^{2,-}u(x)$, with { $\xi \neq 0$}, we have
$$
u\left( x+t \frac{\xi}{|\xi|} \right) \leq \Lambda t + o(t), \quad { \mbox{as } t \to 0^+}.
$$
\end{itemize}

{If a continuous function $u$ is both a viscosity subsolution and a viscosity supersolution we say $u$ is a \textit{viscosity solution of} \eqref{eq} in $\Omega$.}
\end{definition} 

\bigskip

%
%
%
%

\begin{remark}
The equation is interpreted in the usual way in the context of the infinity Laplacian. Now, if $x \in \partial \{ u > 0 \}$ is a point of differentiability and, say, {$(\xi, M) \in {J}_{\Omega}^{2,-}u(x)$},
then
$$ 
  u_\nu^+ = \lim_{t\to 0^+} \frac{u \left ( x + t \frac{\xi}{|\xi|}\right ) }{t} = \nabla u(x) \cdot \frac{\xi}{|\xi|} = |\xi|=|\nabla u(x)|. 
$$
On the other hand,  free boundary condition (ii), along with the subjet estimate, gives
$$
	 t  |\xi| + \frac12 t^2 \left  \langle M \frac{\xi}{|\xi|},  \frac{\xi}{|\xi|}  \right \rangle + o\left(t^2\right)  \le u\left( x+t\frac{\xi}{|\xi|} \right) \leq \Lambda t + o(t).
$$
Dividing the above inequality by $t$ and letting $t\to 0^+$ yields
$$ 
	{ |\nabla u(x)| =|\xi| \leq \Lambda}.
$$
Thus, the interpretation of the free boundary condition given above is a (very) weak representative of the corresponding flux balance in \eqref{eq}.
\end{remark}

Hereafter in this paper, we denote by $B_r(x)$ the euclidean $n$-dimensional ball with radius $r>0$ centered at $x \in \mathbb{R}^n$. By simplicity, we also denote $B_r:=B_r(0)$.

We can now state the main theorem of this article, the optimal regularity for viscosity solutions of \eqref{eq}. 

\begin{theorem}[Lipschitz regularity]\label{mainthm} 

Any bounded viscosity solution $u$ of \eqref{eq}, in the sense of Definition \ref{viscositysol}, is locally Lipschitz continuous. Moreover, for any subdomain $\Omega' \Subset \Omega$, there exist universal constants $C>0$, depending only on $n$ and $\dist(\Omega',\partial\Omega)$, and $L >1$, depending only on $\|f_+\|_{\infty}$, $\|f_-\|_{\infty}$ and $\Lambda$,  such that 
$$\sup_{x,y \in \Omega'} \frac{\left| u(x)-u(y) \right|}{|x-y|} \leq C \left( L + \|u\|_{L^\infty(\Omega)} \right) .$$
\end{theorem}

\section{Pointwise estimates for interior maxima} 

In this section we start preparing for the proof of Theorem \ref{mainthm}, by deriving pointwise estimates involving the intrinsic structure of the infinity Laplacian at interior maximum points of a certain continuous function. Such a powerful analytic tool will be used, so to speak, as a surrogate for the absence of a monotonicity formula in this non-variational two-phase free boundary problem.

\begin{lemma}\label{ishii}
Let $v \in C(B_1)$, $0\leq \omega \in C^2(\mathbb{R}^+)$ and set
$$w(x,y):= v(x)-v(y) \quad \mbox{and} \quad \varphi(x,y):=L  \omega(|x-y|)+\varrho \left( |x|^2+|y|^2 \right),$$
with $L,\varrho$ positive constants. If the function $w-\varphi$ attains a maximum at $(x_0,y_0) \in B_{\frac{1}{2}}\times B_{\frac{1}{2}}$, then, for each $\varepsilon>0$, there exist $M_{x},M_{y} \in \mathcal{S} (n)$, such that
\begin{equation}\label{jets1} 
\left( D_x\varphi(x_0,y_0),M_{x} \right) \in \overline{J}^{2,+}_{B_{1/2}}v(x_0),
\end{equation}
\begin{equation}\label{jets2} 
(-D_y\varphi(x_0,y_0),M_{y}) \in \overline{J}^{2,-}_{B_{1/2}}v(y_0),
\end{equation}
and the estimate
$$\left\langle M_{x} D_x\varphi(x_0,y_0),D_x\varphi(x_0,y_0) \right\rangle - \langle M_{{y}}D_y\varphi(x_0,y_0),D_y\varphi(x_0,y_0) \rangle $$
\begin{equation}\label{comp3}
\leq 4\,L\omega''(\rho)\left(L\omega'(\rho)+\varrho \rho \right)^2+16\varrho \left(L^2\omega'(\rho)^2+\varrho^2 \right)
\end{equation}
holds, where $\rho = |x_0 - y_0|$.
\end{lemma} 

\begin{proof}
Under the hypothesis of the lemma, let us consider a local maximum,  $(x_0,y_0) \in B_{\frac{1}{2}}\times B_{\frac{1}{2}}$, of $w-\varphi$. By \cite[Theorem 3.2]{IL}, for each $\varepsilon>0$, there exist matrices $M_{x},M_{y} \in \mathcal{S} (n)$ such that \eqref{jets1} and \eqref{jets2} hold, and
$$
\left(
\begin{array}{cc}
M_{x} & 0 \\
0 & -M_{y} 
\end{array}
\right)
 \leq A+\epsilon A^2
$$
for
$$
A:=\left(
\begin{array}{cc}
M_\omega & -M_\omega \\
-M_\omega & M_\omega 
\end{array}
\right) + 2\varrho \,  I_{2n\times 2n},
$$
where
\begin{eqnarray}
M_\omega & :=  & L\,\omega''(|x_0-y_0|)\frac{(x_0-y_0) \otimes (x_0-y_0)}{|x_0-y_0|^2} \nonumber \\ 
& & +L\frac{\omega'(|x_0-y_0|)}{|x_0-y_0|}\left(I - \frac{(x_0-y_0) \otimes (x_0-y_0)}{|x_0-y_0|^2}\right). \label{hessian} 
\end{eqnarray}
In particular, we have
$$\left\langle M_{x} D_x\varphi(x_0,y_0) ,D_x\varphi(x_0,y_0) \right\rangle - \langle M_{{y}}D_y\varphi(x_0,y_0),D_y\varphi(x_0,y_0) \rangle $$
 $$\leq \left\langle M_\omega (D_x\varphi(x_0,y_0)-D_y\varphi(x_0,y_0)), D_x\varphi(x_0,y_0)-D_y\varphi(x_0,y_0) \right\rangle $$
\begin{equation}\label{matrix} +2\varrho \left( |D_x\varphi(x_0,y_0)|^2+|D_y\varphi(x_0,y_0)|^2 \right) + \epsilon \lambda,
\end{equation}
where 
$$ \lambda:= \left\langle A^2 \left( D_x\varphi(x_0,y_0), D_y\varphi(x_0,y_0) \right) , \left( D_x\varphi(x_0,y_0), D_y\varphi(x_0,y_0) \right) \right\rangle.$$

Now, for $\nu:=\frac{x_0-y_0}{|x_0-y_0|}$, we have
\begin{equation}\label{grad1}
D_x\varphi(x_0,y_0) =  L \, \omega'(\rho) \nu + 2 \varrho x_0 
\end{equation}
and
\begin{equation}\label{grad2}
 -D_y\varphi(x_0,y_0) = L \, \omega'(\rho) \nu - 2\varrho y_0,
\end{equation}
and thus, with $\iota=2(L\omega'(\rho)\rho^{-1}+\varrho)$, we have
$$D_x \varphi(x_0,y_0)-D_y \varphi(x_0,y_0)=\iota (x_0-y_0).$$
It then follows from \eqref{hessian} that
$$\left\langle M_\omega (D_x\varphi(x_0,y_0)-D_y\varphi(x_0,y_0)), D_x\varphi(x_0,y_0)-D_y\varphi(x_0,y_0) \right\rangle  $$
$$=  \iota^2 \langle M_\omega (x_0-y_0),(x_0-y_0) \rangle= \iota^2 L\omega''(\rho)\rho^2$$
\begin{equation}\label{comp4}
=  4L \omega''(\rho)\left(L\omega'(\rho)+\varrho \rho \right)^2.
\end{equation} 
Moreover, observe that
\begin{eqnarray*}
|D_x\varphi(x_0,y_0)|^2 + |D_y\varphi(x_0,y_0)|^2  = 2L^2\omega'(\rho)^2 & + & 4L\varrho\omega'(\rho)\rho \\
& + & 4\varrho^2(|x_0|^2 +|y_0|^2).
\end{eqnarray*}
Using Cauchy's inequality, we obtain the estimate
$$
4L\varrho\omega'(\rho)\rho \leq \frac{(2L\omega'(\rho))^2}{2}+\frac{(2\varrho \rho)^2}{2} = 2L^2\omega'(\rho)^2+2\varrho^2\rho^2
$$
and then
\begin{eqnarray*}
|D_x\varphi(x_0,y_0)|^2 + |D_y\varphi(x_0,y_0)|^2  \leq 4L^2\omega'(\rho)^2 & + &  2\varrho^2\rho^2 \\
& + & 4\varrho^2(|x_0|^2 +|y_0|^2).
\end{eqnarray*}
Since $\max\{|x_0|,|y_0|,\rho\} \leq 1/2$, we obtain
\begin{equation}\label{comp6}
|D_x\varphi(x_0,y_0)|^2+|D_y\varphi(x_0,y_0)|^2  \leq  4(L^{\,2}\omega'(\rho)^2+\varrho^2).
\end{equation} 

Finally, if $\lambda > 0$, choose 
$$\epsilon = \frac{8\varrho \left(L^2\omega'(\rho)^2+\varrho^2 \right)}{\lambda},$$
otherwise choose $\epsilon$ freely. 
Using \eqref{comp4} and \eqref{comp6} in \eqref{matrix}, together with this choice of $\epsilon$, we obtain \eqref{comp3} and the proof is complete.

\end{proof}

\section{Building an appropriate barrier}

In this section, we derive an ordinary differential estimate which will be used to derive geometric properties related to  problem \eqref{eq}. For positive constants ${\kappa}$ and $\theta$, to be chosen later, we consider the barrier function 
\begin{equation}\label{barrierfunction}
\omega(t)=t-{\kappa} \,t^{1+\theta} \quad \mbox{for} \quad 0<t <1.
\end{equation}

\begin{proposition}\label{barrier}
Let $a$ and $b$ be positive parameters. Given $K>0$, there exist positive constants $\overline L$, $\kappa$ and $\theta$, depending only on $K$, the parameters $a$ and $b$, and universal constants, such that 
\begin{equation}\label{keq}
{ aL^3 \omega''(t)\omega'(t)^2 + bL^2 \omega'(t)^2 < -K},
\end{equation}
for all $L \geq \overline L$. Moreover, there holds
\begin{equation}\label{keq2}
\omega(t)>0, \quad \dfrac{1}{2} \leq \omega'(t) \leq 1 \quad \mbox{and} \quad \omega''(t) < 0,
\end{equation}
for any $0<t<1$. 
\end{proposition}

\begin{proof}
By direct computation, one obtains
\begin{equation}\nonumber
-\omega''(t)\,\omega'(t)^2 =  
\kappa(1+\theta)\theta \left( t^{\theta-1}-2\kappa(1+\theta)t^{2\theta-1}+\kappa^2(1+\theta)^2t^{3\theta-1} \right).
\end{equation}
Hence, by choosing (and fixing hereafter) $1/2 \leq \theta \leq 1$, we obtain
\begin{equation}\nonumber
\begin{array}{ccl}
-\omega''(t)\,\omega'(t)^2 & \geq & 
\kappa(1+\theta)\theta \left( 1-2\kappa(1+\theta)\right) \\[0.2cm]
 & \geq & \dfrac{4\kappa}{3} \left( 1-4\kappa\right)=: \overline{\kappa}>0,
\end{array}
\end{equation}
provided $\kappa < 1/4$. In view of this and $\omega'(t) \leq 1$, we obtain  
$$
{aL^3 \omega''(t)\omega'(t)^2 + bL^2 \omega'(t)^2 < -a\overline\kappa L^3+bL^2.}
$$
Then, we select $\overline L$ large such that estimate \eqref{keq} holds for every $L\geq \overline L$. The first and third estimates in \eqref{keq2} follow immediately. We conclude the proof by observing that
$$
\omega'(t) \geq 1-\kappa(1+\theta) \geq 1-2\kappa \geq \frac{1}{2}. 
$$

\end{proof}

\section{Proof of the main Theorem} 

In this final section, we prove Theorem \ref{mainthm}. For simplicity, we take $\Omega=B_1$ and $\Omega' = B_{1/2}$. The strategy is to show that for some 
$$
	L=L \left( \|f_+\|_{L^\infty(B_1)},\|f_-\|_{L^\infty(B_1)} , \Lambda \right) \gg 1 \quad \text{ and } \quad \varrho=\varrho \left( \|u\|_{L^\infty(B_1)} \right) >0,
$$ 
to be chosen later, and for any $z_0 \in B_{1/2}$ fixed, there must hold
\begin{equation}\label{comp0}
	\sup\limits_{(x,y) \in B_{1/2}(z_0) \times B_{1/2}(z_0)}  \left[ u(x) - u(y) \right] \le L  \omega(|x-y|) + \varrho \left( |x-z_0|^2+|y-z_0|^2 \right).
\end{equation}
Estimate \eqref{comp0} clearly implies that $u$ is $\left (L + \varrho \right )$-Lipschitz continuous at $z_0$. For simplicity, hereafter in the proof, let us take $z_0 = 0$. 

We will achieve \eqref{comp0} by proving that the existence of a pair of points  $(x_0,y_0) \in \overline{B_{1/2}} \times \overline{B_{1/2}}$ verifying
\begin{equation}\label{comp1}
u(x_0)-u(y_0)-L  \omega(|x_0-y_0|)-\varrho \left( |x_0|^2+|y_0|^2 \right) > 0
\end{equation}
enforces a universal limitation upon the constant $L$. Hence, the reasoning starts by assuming \eqref{comp1}, which readily implies that $x_0 \neq y_0$ and that
\begin{equation}\label{intest}
\varrho  \left( |x_0|^2+|y_0|^2 \right) \leq 2 \|u\|_{L^{\infty}(B_1)}.
\end{equation}
Thus, in order to guarantee that $x_0, y_0$ are interior points in $B_{1/2}$, we just need to select, once and for all, 
$$
\varrho := 9  \|u\|_{L^{\infty}(B_1)}.
$$

Next, we note that $\omega$ is twice continuously differentiable in a small neighborhood of $\eta:=|x_0-y_0|>0$, and thus Lemma \ref{ishii} guarantees the existence of 
$$
(\xi_x, M_x) \in \overline{J}^{2,+}_{B_{1/2}}u(x_0) \qquad {\rm and} \qquad (\xi_y, M_y) \in \overline{J}^{2,-}_{B_{1/2}}u(y_0)
$$ 
satisfying
\begin{equation}\label{finaleq1}
\langle M_x\xi_x,\xi_x\rangle - \langle M_y\xi_y,\xi_y \rangle \leq aL^3  \omega''(\eta)\omega'(\eta)^2+b\,L^2  \omega'(\eta)^2+d,
\end{equation}
for universal positive parameters $a,b$ and $d$. We have further used the fact that $\omega'' (\eta)<0$. Hence, by \eqref{finaleq1} and Proposition \ref{barrier}, given $K>0$ there exists $\overline L \gg 1$, such that 
\begin{equation}\label{finaleq2}
\langle M_x\xi_x,\xi_x\rangle - \langle M_y\xi_y,\xi_y \rangle <-K,
\end{equation}
for all $L \gg \overline L$.

\medskip
In what follows, we want to prove that 
$$
x_0 \in \mathcal{R}(u) \cup \{u>0\} \quad \mbox{and} \quad y_0 \in \mathcal{R}(u) \cup \{u<0\},$$
and, in addition, that
$$
\{x_0,y_0\} \cap \mathcal{R}(u) \neq \emptyset \quad \mbox{and} \quad \{x_0,y_0\} \cap \mathcal{R}(u) \neq \{x_0,y_0\}.
$$
 
For that purpose, we initially note that \eqref{comp1} yields
\begin{equation}\label{finalest2}
u(x_0)-u(y_0) > 0.
\end{equation}
Hence, if $x_0$ were to be in $\{u<0\}$, then  $y_0$ would necessarily also belong to $\{u<0\}$. However, combining the fact that $u$ solves $-\Delta_\infty u = f_-$ in its negative phase with \eqref{finaleq2}, we should have
\begin{equation}\label{finaleqA}
-f_-(x_0) + f_-(y_0) <-K, 
\end{equation} 
which yields a contradiction by choosing $\overline L$ universally large such that $K \geq 4\|f_-\|_{L^\infty(B_1)}$. 

Arguing similarly, if one assumes $y_0 \in \{u>0\}$, then $x_0$ would also have to be in $\{u>0\}$, and the same reasoning employed above would lead to a contradiction, choosing $\overline L$ such that $K \geq 4\|f_+\|_{L^\infty(B_1)}$. Likewise, for $K \geq 4 \max\{\|f_+\|_{L^\infty(B_1)},\|f_-\|_{L^\infty(B_1)}\}$, the case $\{x_0,y_0\} \cap \mathcal{R}(u) = \emptyset$ is ruled out. Finally, from \eqref{finalest2} we easily conclude that $\{x_0,y_0\} \cap \mathcal{R}(u) \neq \{x_0,y_0\}$. 

\medskip

We are now left with two cases to investigate. The following picture gives an enlightening view of the subsequent analysis. 

\bigskip
\medskip

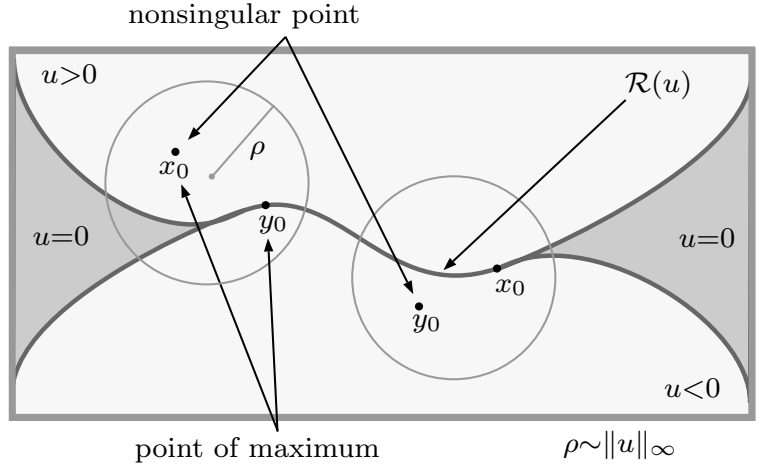
\centerline{\psscalebox{1.35 1.35} 
{
\begin{pspicture}(0,-2.2265325)(7.3544188,2.2265325)
\definecolor{colour0}{rgb}{0.6,0.6,0.6}
\definecolor{colour1}{rgb}{0.96862745,0.96862745,0.96862745}
\definecolor{colour2}{rgb}{0.4,0.4,0.4}
\definecolor{colour3}{rgb}{0.8,0.8,0.8}
\psframe[linecolor=colour0, linewidth=0.062, fillstyle=solid,fillcolor=colour1, dimen=outer](7.3471565,1.8745947)(0.004166668,-1.8087386)
\psbezier[linecolor=colour2, linewidth=0.044, fillstyle=solid,fillcolor=colour3](7.300969,-1.6347036)(7.302879,-0.8530342)(5.803756,0.084505685)(4.9398313,-0.26150033186541466)(4.0759068,-0.60750633)(7.330422,0.52590984)(7.319385,1.6374248)(7.308348,2.7489398)(7.299059,-2.416373)(7.300969,-1.6347036)
\rput[bl](0.31798726,1.4924241){$\scriptstyle u>0$}
\rput[bl](6.4556174,-1.5959479){$\scriptstyle u<0$}
\rput[bl](1.4776772,0.5817708){$\scriptstyle x_0$}
\psdots[linecolor=black, dotsize=0.08](1.6411545,0.83608973)
\psdots[linecolor=black, dotsize=0.08](4.0410686,-0.6892136)
\rput[bl](2.4587624,0.034063153){$\scriptstyle y_0$}
\rput[bl](4.78874,-0.57936984){$\scriptstyle x_0$}
\psline[linecolor=colour0, linewidth=0.02](2.0178986,0.60935664)(2.6149197,1.2948052)
\psdots[linecolor=colour0, dotsize=0.06](1.9985963,0.59268993)
\rput[bl](3.9725807,-0.9434439){$\scriptstyle y_0$}
\rput[bl](2.3710437,0.754938){$\scriptstyle \rho$}
\rput[bl](1.2421536,-2.2265325){\scriptsize{point of maximum}}
\rput[bl](1.174706,2.0465324){\scriptsize{nonsingular point}}
\rput[bl](6.0717335,1.3530458){$\scriptstyle \mathcal{R}(u)$}
\rput[bl](5.458322,-2.1942923){$\scriptstyle \rho \sim \|u\|_\infty$}
\psbezier[linecolor=colour2, linewidth=0.044, fillstyle=solid,fillcolor=colour3](0.049695976,1.7603675)(0.04780361,0.9812071)(1.2913101,-0.16196814)(2.1473076,0.18292727972971648)(3.0033052,0.5278227)(0.020513104,-0.4607752)(0.031448863,-1.5687225)(0.04238462,-2.6766698)(0.05158834,2.539528)(0.049695976,1.7603675)
\psbezier[linecolor=colour2, linewidth=0.044](2.2963612,0.24951221)(3.251013,0.6145931)(3.6785223,-0.68896157)(4.8003044,-0.3166123067602382)
\psdots[linecolor=black, dotsize=0.08](2.530623,0.30965564)
\psdots[linecolor=black, dotsize=0.08](4.808275,-0.31547174)
\pscircle[linecolor=colour0, linewidth=0.02, dimen=outer](4.382611,-0.40575412){1.010238}
\psframe[linecolor=colour0, linewidth=0.068, dimen=outer](7.3544188,1.8662614)(0.0,-1.8170719)
\psline[linecolor=black, linewidth=0.02, arrowsize=0.05291667cm 2.0,arrowlength=1.4,arrowinset=0.0]{->}(2.6477797,-1.9125177)(1.6989175,0.5162835)
\rput[bl](6.599689,-0.10109619){$\scriptstyle u=0$}
\pscircle[linecolor=colour0, linewidth=0.02, dimen=outer](1.9521788,0.5291296){1.010238}
\psline[linecolor=black, linewidth=0.02, arrowsize=0.05291667cm 2.0,arrowlength=1.4,arrowinset=0.0]{->}(6.0755367,1.3468535)(4.2804437,-0.33144012)
\psline[linecolor=black, linewidth=0.02, arrowsize=0.05291667cm 2.0,arrowlength=1.4,arrowinset=0.0]{->}(2.7244823,1.9707191)(3.9837909,-0.5671464)
\rput[bl](0.23799834,-0.081950136){$\scriptstyle u=0$}
\psline[linecolor=black, linewidth=0.02, arrowsize=0.05291667cm 2.0,arrowlength=1.4,arrowinset=0.0]{->}(2.726334,1.9682157)(1.7641648,0.9526795)
\psline[linecolor=black, linewidth=0.02, arrowsize=0.05291667cm 2.0,arrowlength=1.4,arrowinset=0.0]{->}(2.6470819,-1.9228915)(2.5530076,-0.01657915)
\end{pspicture}
}}

\bigskip
\medskip

\noindent {\bf Case 1.} Suppose $x_0 \in \{u>0\}$ and $y_0 \in \mathcal{R}(u)$. From this and estimate \eqref{finaleq2}, we note that
\begin{equation}\label{coimbra} 
-\langle M_y \xi_y,\xi_y \rangle < -K + \|f_+\|_{L^{\infty}(B_1)} \leq -\frac K 2 < 0,
\end{equation}
according to our previous choice for $K$.  

Since $(\xi_y,M_y) \in \overline{J}^{2,-}u(y_0)$, we are not able to apply directly the free boundary condition given in Definition \ref{viscositysol}. In order to address this issue, we take sequences 
$$
y_k \to y_0, \quad \xi_k \to \xi_y \quad \mbox{and} \quad M_k \to M_y   \qquad (\mbox{as } k\to \infty),
$$
such that $(\xi_k,M_k) \in {J}^{2,-}u(y_k)$. We have
\begin{equation}\label{jampa}
	u(x) \geq \left\langle \xi_k , x-y_k \right\rangle + \frac{1}{2} \left\langle M_k (x-y_k), (x-y_k)\right\rangle + o \left( |x-y_k|^2 \right),
\end{equation} 
for $x$ sufficiently close to $y_k$. Now, we claim $y_k \in \mathcal{R}(u)$. Suppose otherwise that $y_k \in \{u>0\} \cup \{u<0\}$. Then, by Definition \ref{viscositysol}, stability and \eqref{coimbra}, estimate
\begin{equation}\label{orlando} 
\min\{f_+(y_k),f_-(y_k)\} \leq -\langle M_{k} \xi_k,\xi_k \rangle \leq -\frac K 4
\end{equation}
would hold for $k\gg 1$. By taking $K \gg \max\{\|f_+\|_{\infty},\|f_-\|_{\infty}\}$, we reach a contradiction.

Also, notice that, by \eqref{coimbra}, we have $\xi_k \neq 0$.
Therefore, applying the free boundary condition at $y_k$, together with \eqref{jampa} and \eqref{coimbra}, we obtain, for points $x=y_k + t\, \xi_k/|\xi_k|$, the estimates
\begin{equation}\nonumber
\begin{array}{rcl}
\Lambda t+o(t) & \geq &  t|\xi_k|+ \dfrac{t^2}{2|\xi_k|^2}\langle M_k \xi_k,\xi_k\rangle  + o(t^2) \\[0.45cm]
 & \geq & t|\xi_k|+ o(t^2),
\end{array}
\end{equation}
for each $t>0$ sufficiently small. Finally, dividing by $t$, letting $t \to 0$ and subsequently $k \to \infty$, we get
\begin{equation}\label{fbcond}
\Lambda \geq |\xi_y|.
\end{equation}

On the other hand, from \eqref{grad2}, we obtain 
$$
	\xi_y = L\, \omega'(\eta)\eta^{-1}(x_0-y_0)-2\varrho \,y_0,
$$ 
and hence, from estimate \eqref{keq2}, we know there holds 
\begin{equation}\label{lastone}
|\xi_y| \geq \frac{L}{2}-2\varrho.
\end{equation} 
Thus, plugging \eqref{lastone} into \eqref{fbcond}, we reach to a contradiction by choosing $L$ universally large, now depending also on $\Lambda$.  
 
\medskip 

\noindent {\bf Case 2.} Suppose, alternatively, that  $y_0 \in \{u<0\}$ and $x_0 \in \mathcal{R}(u)$. In this case, inequality \eqref{finaleq2} provides  
\begin{equation}\label{finalest3a}
\langle M_x \xi_x,\xi_x \rangle < -K + \|f_-\|_{L^{\infty}(B_1)} \leq -\frac K 2 <0, 
\end{equation}
for $K$ sufficiently large. Since $(\xi_x,M_x) \in \overline{J}^{2,+}u(x_0)$, we guarantee the existence of sequences
$$ 
x_j \to x_0, \quad \xi_j \to \xi_x \quad \mbox{and} \quad M_j \to M_x     \qquad (\mbox{as } j\to \infty),
$$
such that $(\xi_j,M_j) \in {J}^{2,+}u(x_j)$. Then
\begin{equation}\label{jampa2}
	u(x) \leq \left\langle \xi_j , x-x_j \right\rangle + \frac{1}{2} \left\langle M_j (x-x_j), (x-x_j)\right\rangle + o \left(|x-x_j|^2 \right),
\end{equation}
for any $x$ sufficiently close to $x_j$. Similarly to case 1, we have $x_j \in \mathcal{R}(u)$, for $j \gg 1$. Indeed, if $u(x_j)\neq 0$, we would have, by Definition \ref{viscositysol}, stability and \eqref{finalest3a} that
\begin{equation}\label{orlando2}
\min\{-f_+(x_j),-f_-(x_j)\} \leq \langle M_j \xi_j,\xi_j \rangle \leq -\frac K 4,
\end{equation}
reaching a contradiction by taking $K\gg \max\{\|f_+\|_\infty,\|f_-\|_\infty\}$. 

Still as in case 1, we observe from \eqref{finalest3a} that $\xi_j \neq 0$. Therefore, from Definition \ref{viscositysol}, \eqref{jampa2} and \eqref{finalest3a}, we have, for $x=x_j - t\, \xi_j/|\xi_j|$,
\begin{equation}\nonumber
\begin{array}{rcl}
	-\Lambda t+o(t) & \leq &   - t|\xi_j|+ \dfrac{t^2}{2|\xi_j|^2}\langle M_j \xi_j,\xi_j \rangle + o(t^2) \\[0.45cm] 
	& \leq & - t|\xi_j|+ o(t^2), \\
\end{array}
\end{equation}
for $t>0$ sufficiently small. Dividing by $t$, letting $t \to 0$ and then $j \to \infty$, we conclude that $|\xi_x| \leq \Lambda$.  

On the other hand, from \eqref{grad1} we have 
 $$
 	\xi_x= L\, \omega'(\eta)\eta^{-1}(x_0-y_0)+2\varrho \,x_0
$$  
and thus, from Proposition \ref{barrier}, we estimate 
\begin{equation}\label{finalest31}
	\Lambda \geq |\xi_x| \geq \frac{L}{2}-2\varrho.
\end{equation}
Finally, by taking $L$ universally large,
we again reach a contradiction. Thus \eqref{comp1} can not hold and the proof of Theorem \ref{mainthm} is complete. 
 
\qed

\bigskip

{\small \noindent{\bf Acknowledgments.} DJA supported by CNPq grant 427070/2016-3 and grant 2019/0014 from Para\'\i ba State Research Foundation (FAPESQ). 

JMU partially supported by FCT -- Funda\c c\~ao para a Ci\^encia e a Tecnologia, I.P., through grant SFRH/BSAB/150308/2019 and projects PTDC/MAT-PUR/28686/2017 and UTAP-EXPL/MAT/0017/2017, and by the Centre for Mathematics of the University of Coimbra - UIDB/00324/2020, funded by the Portuguese Government through FCT/MCTES.

DJA and JMU thank the hospitality of the University of Central Florida, and DJA thanks the Abdus Salam International Centre for Theoretical Physics, where parts of this work were conducted.}

\bigskip

\bibliographystyle{amsplain, amsalpha}

\end{document}